\documentclass[a4paper,12pt]{article}

\usepackage[T1]{fontenc}
\usepackage{lmodern}
\usepackage[utf8]{inputenc}

\usepackage{amsmath, amssymb, amsthm, mathtools} 

\usepackage{graphicx} 
\usepackage{pgfplots} 

\usepackage[authoryear,round]{natbib}

\usepackage[pdftex]{hyperref}
\hypersetup{plainpages=True, pdfstartview=FitV,colorlinks=true,linkcolor=blue,citecolor=blue}

\newtheorem{thm}{Theorem}[section]
\newtheorem{cor}[thm]{Corollary}
\newtheorem{lem}[thm]{Lemma}
\newtheorem{prop}[thm]{Proposition}
\theoremstyle{remark}

\newtheorem{exm}[thm]{Example}
\author{Vytas Zacharovas\thanks{Institute of Statistical Sciences, Academia Sinica, Taipei, Taiwan}\thanks{Institute of Computer Science, Vilnius University, Vilnius, Lithuania, E-mail:vytas.zacharovas@mif.vu.lt}\thanks{This work was supported in part by the National Science and Technology Council (NSTC) of Taiwan under Grant No. NSTC 112-2811-M-001-002.}}

\title{Convergence in \(\chi^2\) Distance to the Normal Distribution for Sums of Independent Random Variables}

\begin{document}
\maketitle
\begin{abstract}
Suppose \( n \) independent random variables \( X_1, X_2, \dots, X_n \) have zero mean and equal variance. We prove that if the average of \(\chi^2\) distances between these variables and the normal distribution is bounded by a sufficiently small constant, then the \(\chi^2\) distance between their normalized sum and the normal distribution is \( O(1/n) \).
\end{abstract}



\section{Introduction}
Suppose 
\[
S_n=\frac{X_1+X_2+\cdots+X_n}{\sqrt{n}}
\]
	where \(X_j\) are independent, \(\mathbb{E}X_j=0\) and \(\mathbb{E}X_j^2=1\). One of the earliest results in probability theory is the central limit theorem stating that the distribution of this sum converges to standard normal distribution  with the density
\[
\varphi(x)=\frac{e^{-x^2/2}}{\sqrt{2\pi}},
\]
if fairly general conditions are satisfied. The closeness  of the distribution of \(S_n\) and the normal distribution can be quantified in a number of different metrics such as total variation distance, Kolmogorov distance, etc. This work is devoted to the study of the so called \(\chi^2\) distance between these two distributions. 
For a random variable \(Y\)  with a density \(p(x)\) the \(\chi^2\) distance between the distribution of \(Y\) and that of the standard normal distribution is defined as 
\[
\chi^2(Y,\mathcal{N})=\int ^{+\infty }_{-\infty }\left( \dfrac{p\left( x\right) }{\varphi \left( x\right) }-1\right) ^{2}\varphi \left( x\right)\, dx,
\]
where \(\varphi(x)\) is the density of the standard normal distribution. This distance is important as it serves as an upper bound for several other distances. For example, a simple application of Cauchy inequality shows that it is an upper bound for total variation and Kolmogorov distances
\[
\sup_{y\in \mathbb{R}}|\mathbb{P}(Y\leqslant y)-\Phi(y)|\leqslant\frac{1}{2}\int ^{+\infty }_{-\infty }|p(x)-\varphi \left( x\right)|\, dx\leqslant \frac{1}{2}\sqrt{\chi^2(Y,\mathcal{N})},
\]
where \(\Phi(x)=\int_{-\infty}^x\phi(t)\,dt\) is the distribution function of the standard normal distribution. A trivial application of the elementary inequality \(\log x\leqslant x-1\) shows that the \(\chi^2\)-distance is an upper bound for the information divergence distance
\[
\int ^{+\infty }_{-\infty }p(x)\log\frac{p(x)}{\varphi \left( x\right)}\, dx\leqslant \chi^2(Y,\mathcal{N}).
\] 
Once established, the upper bound for \(\chi^2(Y,\mathcal{N})\) leads to a non-uniform estimate of difference of distribution functions with very fast decreasing tails
\[
|\mathbb{P}(Y\leqslant y)-\Phi(y)|\leqslant \sqrt{\min\{\Phi(y),1-\Phi(y)\}}\sqrt{\chi^2(Y,\mathcal{N})}.
\]
For more information on the relation of \(\chi^2\)-distance with other metrics we refer the reader to \cite{bobkov_et_al_2019}.

 The earliest result evaluating \(\chi^2(S_n,\mathcal{N})\) is the paper by \cite{fomin_1982} who showed that if all the variables \(X_j\) have the same distribution and a certain condition on the moments of \(X_j\) is satisfied, then 
\[
\chi^2(S_n,\mathcal{N})=O\left(\frac{1}{n}\right).
\]
The condition of Fomin's paper was fairly restrictive. Recently \cite{bobkov_et_al_2019} proved necessary and sufficient conditions under which a sum of \(n\) identically distributed random variables converges to normal distribution in \(\chi^2\) distance.
\begin{thm}[\citet*{bobkov_et_al_2019}]
\label{thm_bobkov_et_al} Suppose all the random variables \(X_j\) are identically distributed. Then we have \(\chi^2(S_n,\mathcal{N})\to 0\) as \(n\to\infty\), if and only if \(\chi^2(S_n,\mathcal{N})\) is finite for some \(n=n_0\), and
\begin{equation}
\label{sub-gausian_condition}
\mathbb{E}e^{tX_1}<e^{t^2}\quad\text{for all real }t\not=0.	
\end{equation}
In this case, \(\chi^2\)-distance admits and Edgeworth-type expansion
\begin{equation}
	\chi^2(S_n,\mathcal{N})=\sum_{j=1}^{s-2}\frac{c_j}{n^j}+O\left(\frac{1}{n^{s-1}}\right)
\end{equation}
which is valid for every \(s=3,4,\ldots\) with coefficients \(c_j\) representing certain polynomials in the moments \(\mathbb{E}X^k\), \(k=3,\ldots,j+2\).
\end{thm}
The main result of this paper is the following theorem.
\begin{thm}
\label{main_corollary}
 Suppose \(X_1,X_2,\ldots,X_n\) with \(n\geqslant 2\) are independent random variables such that \(\mathbb{E}X_j=0\) and \(\mathbb{E}X_j^2=1\)  for all \(j=1,2,\ldots,n\).
 There exists an absolute constant \(C>0\) such that if
\[
\frac{1}{n}\sum_{j=1}^{n}\chi^2(X_{j},\mathcal{N})\leqslant 0.82
\]
then 
\[
\begin{split}
\chi^2(S_n,\mathcal{N})&\leqslant \frac 1{n-1}\left(\frac{1}{n}\sum_{j=1}^{n}\chi^2(X_{j},\mathcal{N})+C\left(\frac{1}{n}\sum_{j=1}^{n}\chi^2(X_{j},\mathcal{N})\right)^2\right).
\end{split}
\]
If   the independent random variables \(X_1,X_2,\ldots,X_n\) are symmetric (that is \(-X_j\) and \(X_j\) have the same distribution) in addition to having zero mean and variance equal to one,   then there exists an absolute constant \(C_1>0\) that if
\[
\frac{1}{n}\sum_{j=1}^{n}\chi^2(X_{j},\mathcal{N})\leqslant 1.69
\]
then  
\[
\begin{split}
\chi^2(S_n,\mathcal{N})&\leqslant \frac {1}{n^2-1}\left(\frac{1}{n}\sum_{j=1}^{n}\chi^2(X_{j},\mathcal{N})+C_1\left(\frac{1}{n}\sum_{j=1}^{n}\chi^2(X_{j},\mathcal{N})\right)^2\right).
\end{split}
\]
The constants \(C,C_1\) are absolute.
\end{thm}
In other words, we prove that if the distributions of \( X_j \) are, on average, sufficiently close to the standard normal distribution in \(\chi^2\) distance, then the \(\chi^2\) distance between the normalized sum \( S_n \) and the normal distribution is \( O(1/n) \).  The inequalities of the theorem are a consequence of  more precise estimates of our Theorem \ref{main_theorem} from which constants \(C,C_1\) could be readily evaluated. It is important to note that the constants \(0.82\) and \(1.69\) in our Theorem \ref{main_corollary} are not optimal.  Finding the optimal constants would be the aim of further research in this direction.

In the case when the summands of \(S_n\) are identically distributed our theorem means   that if  \(\chi^2(X_{1},\mathcal{N})\leqslant0.82\)   then \(\chi^2(S_n,\mathcal{N})=O(1/n)\), which according to Theorem \ref{thm_bobkov_et_al} implies that \(X_1\) satisfies the subgausian condition   (\ref{sub-gausian_condition}) of \cite{bobkov_et_al_2019}. Indeed, independent of our main result, we prove (see Theorem \ref{thm_subgausian_intro}) that a random variable \( Y \) with zero mean and unit variance satisfies the subgaussian condition if its \(\chi^2\) distance from the standard normal distribution does not exceed \( 0.96 \). If the variable \( Y \) is symmetric with unit variance, we show that \( \chi^2(Y, \mathcal{N}) < 1.97 \) is sufficient for the subgaussian condition (\ref{sub-gausian_condition}) to hold. See \cite{bobkov_et_al_2024} for further references on subgaussian distributions.

The proofs in \cite{bobkov_et_al_2019} rely heavily on the analysis of characteristic functions. The work of \cite{fomin_1982} was based on analysis of  the Parseval's identity
\begin{equation}
\label{parceval_identity_for_sum}
\chi^2(S_n,\mathcal{N})=\sum_{m=1}^\infty\frac{\bigl(\mathbb{E}H_m(S_n)\bigr)^2}{m!},	
\end{equation}
where \(H_n(x)\) are Hermite polynomials of \(n\)-th order. His approach relied on analysis of certain recurrence identities that quantities \(\mathbb{E}H_m(S_n)\) inside the Parseval's identity (\ref{parceval_identity_for_sum}) satisfy. Our approach is closer to the original Fomin's approach with   ideas from Stein's approach involved in the crucial step of our analysis. Specifically, Stein’s approach to assessing the closeness of a random variable \( S_n \) to the standard normal distribution involves bounding expectations of the form
\[
\mathbb{E}g(S_n),
\]
where \( g \) is a test function satisfying
\[
\int_{-\infty}^\infty g(t)\varphi(t)\,dt = 0.
\]
The core idea is to express \( g \) via \emph{Stein’s equation}:
\begin{equation*}
g(x) = x f(x) - f'(x),
\end{equation*}
so that
\begin{equation}
\label{steins_equation}
\mathbb{E}g(S_n) = \mathbb{E} \bigl( S_n f(S_n) - f'(S_n) \bigr).
\end{equation}
The right-hand side of this expression can often be bounded more easily using properties of the summands of \( S_n \), such as independence, weak dependence, or control of covariances.
A key insight in our work is the observation that Hermite polynomials satisfy a recurrence identity analogous to Stein’s equation:
\[
H_{n+1}(x) = x H_n(x) - H_n'(x),
\]
which immediately implies
\[
\mathbb{E}H_{m+1}(S_n) = \mathbb{E} \bigl( S_n H_m(S_n) - H_m'(S_n) \bigr).
\]
The right-hand side is structurally identical to that in Stein’s identity~\eqref{steins_equation}. This identity allows us to derive recurrence relations for \( \mathbb{E} H_m(S_n) \) that  lead to sharper upper bounds for the \( \chi^2 \)-distance than those obtained by more direct approach that was used by Fomin. Our approach follows the general framework of  paper by \cite{zacharovas_2023} which also attempted to adapt the ideas of Stein's approach to the analysis of \(\chi^2\) distance in the context of Poisson approximation. In the context of Poisson approximation, we successfully obtained estimates for the \( \chi^2 \)-distance that are comparable to those previously derived by complex-analytic methods by \cite{zacharovas_hwang_2010}.

\section{Proofs}
\subsection{\(\chi^2\) distance and  Hermite polynomials} In this section we will summarize all the  properties of Hermite polynomials that will be necessary for later presentation. None of the material presented here is original, the aim of this section is to make the later material more accessible to a reader with no previous experience with Hermite  polynomials and their properties.

Hermite polynomials can be defined in several equivalent ways. We prefer to define them as polynomials that satisfy the identity
\begin{equation}
	\label{gen_f_for_hermitian}
\sum_{n=0}^{\infty}\frac{H_n(x)}{n!}z^n=e^{xz-z^2/2}.
\end{equation}
Hence applying the formula for Taylor coefficients of an infinite series \(H_n(x)=\frac{d^n}{dz^n}e^{xz-z^2/2}\Bigl|_{z=0}\) we can easily compute any number of first Hermite polynomials. For example
\[
\begin{split}
H_0(x)=1,\quad
H_1(x)=x,\quad
H_2(x)=x^2-1,\quad
H_3(x)=x^3-3x
\end{split}
\]
and so on.
In what follows we will only need the properties of these polynomials that are summarized in the next theorem.
\begin{thm}
\begin{enumerate}The following properties of Hermitian polynomials are true:
\item (Parseval's identity)
If for the function \(f(x)\)  defined on the whole real line  the integral
\[
\int_{-\infty}^{+\infty}\bigl|f(x)\bigr|^2 \phi(x)\,dx
\]
is finite, then
\begin{equation}
\label{parseval_general}
\int_{-\infty}^{+\infty}\bigl|f(x)\bigr|^2 \phi(x)\,dx=\sum_{n=0}^\infty\frac{1}{n!}\left|
\int_{-\infty}^{+\infty}f(x)H_n(x) \phi(x)\,dx
\right|^2.
\end{equation}

	\item Hermite polynomials of different order are related by the recurrent identity
	\begin{equation}
	\label{hermite_recurrence}
	H_{n+1}(x)=xH_n(x)-H_n'(x).	
	\end{equation}
\item Suppose numbers \(\alpha,\beta\) are such that
	\[
\alpha ^{2}+\beta ^{2}=1,
\]
then the following addition formula for Hermite polynomials is true
\begin{equation}
\label{hermit_sum}
H_{m}\left( x\alpha+y\beta\right)=\sum ^{n}_{k=0}\binom{m}{k}H_{m-k}\left( {x}\right) H_{k}\left( {y}\right) \alpha ^{m-k}\beta ^{k}	
\end{equation}
and if \(\beta\not=0\),  then
\begin{equation}
\label{hermit_sum_diff}
H_{m}'\left( x\alpha+y\beta\right)=\sum ^{n}_{k=0}\binom{m}{k}H_{m-k}\left( {x}\right) H_{k}'\left( {y}\right) \alpha ^{m-k}\beta ^{k-1}.	
\end{equation}

\end{enumerate}
\end{thm}
\begin{proof} All the listed properties with exception of formulas (\ref{hermit_sum}) and (\ref{hermit_sum_diff}) are classical results that can be found in many textbooks on mathematical analysis see e.g. \cite{szego_1975}. Addition formula  (\ref{hermit_sum}) is an easy consequence of the identity 
\[
e^{(x\alpha+y\beta)z-\frac{z^2}2}= e^{ x (\alpha z)-\frac{(\alpha z)^2}2} e^{ x (\beta z)-\frac{(\beta z)^2}2}
\]
between generating functions (\ref{gen_f_for_hermitian}) of Hermite polynomials for variables \(x\), \(y\) and \(x\alpha+y\beta\). For  \(\alpha=\beta=1/\sqrt{2}\) this formula also appears in \cite{szego_1975} and for general \(\alpha,\beta\) it  was already used by \cite{bobkov_et_al_2019} for the analysis of \(\chi^2\) distance. The formula (\ref{hermit_sum_diff}) is obtained by differentiating the  formula (\ref{hermit_sum}) with respect to \(y\) and dividing both sides by \(\beta\).

\end{proof}
Let \( p(x) \) be the density of a random variable \( Y \). If \(\chi^2(Y, \mathcal{N})\) is finite, then the integral  
\[
\int_{-\infty}^{+\infty} \left( \frac{p(x)}{\varphi(x)} \right)^2 \varphi(x) \,dx
\]
is also finite. Consequently, Parseval's identity (\ref{parseval_general}), applied to the function  
\[
f(x) = \frac{p(x)}{\varphi(x)},
\]
yields the identity  
\[
\int_{-\infty}^{+\infty} \left( \frac{p(x)}{\varphi(x)} \right)^2 \varphi(x) \,dx
= \sum_{n=0}^{\infty} \frac{1}{n!} \left|
\int_{-\infty}^{+\infty} p(x) H_n(x) \,dx
\right|^2.
\]
Since \(p(x)\) is a probability density of a random variable \(Y\), the integrals inside the Parceval identity can be expressed as expectations
\[
\int_{-\infty}^{+\infty}p(x)H_n(x)=\mathbb{E}H_n(Y).
\]
Hence taking into account that \(\mathbb{E}H_0(Y)=1\) we can express the \(\chi^2(Y,\mathcal{N})\) in terms of moments of Hermite polynomials of  \(Y\) as 
\[
\chi^2(Y,\mathcal{N})=\int ^{+\infty }_{-\infty } \left( \dfrac{p\left( x\right) }{\varphi\left( x\right) }\right) ^{2}\varphi(x)\,dx-1=\sum_{n=1}^\infty\frac{\bigl(\mathbb{E}H_n(Y)\bigr)^2}{n!}.
\]
Replacing here \(Y=S_n\) we get the Parceval's identity in the form (\ref{parceval_identity_for_sum}) that will be used throughout the paper.

Note that \( H_1(x) = x \) and \( H_2(x) = x^2 - 1 \). Therefore, if the random variable \( Y \) has zero mean (\(\mathbb{E}Y = 0\)) and unit variance (\(\mathbb{E}Y^2 = 1\)), then  
\[
\mathbb{E}H_1(Y) = \mathbb{E}H_2(Y) = 0.
\]
As a result, the summation in Parseval's identity should begin at \( n = 3 \) instead of \( n = 1 \).

\subsection{Recurrent relations for \(\chi^2\) distance}
Before proving our main result on sums of random variables with equal variance, we first examine the more general case where the summands of \( S_n \) may have different variances. Let
\[
S_n=\sigma_1X_1+\sigma_2X_2+\cdots+\sigma_nX_n
\]
where \(X_j\) are independent, \(\mathbb{E}X_j=0\) and \(\mathbb{E}X_j^2=1\) and \(\sigma_1^2+\sigma_2^2+\cdots+\sigma_n^2=1\) with \(\sigma_j^2<1\) for all \(j=1,2,\ldots,n\).
Let us also define
\[
S_{n;k}=\frac{S_n-\sigma_kX_k}{\sqrt{1-\sigma_k^2}}
\]
thus
\[
S_n=\sqrt{1-\sigma_k^2}S_{n;k}+\sigma_kX_k.
\]
To evaluate the right-hand side of Parseval's identity (\ref{parceval_identity_for_sum}), which expresses \(\chi^2(S_n, \mathcal{N})\) in terms of the expectations of Hermite polynomials of different orders, we will use the recurrence relation given in the following theorem.
\begin{thm}
\label{stein_recurrent_identity}
The  identity
\[
\begin{split}
\mathbb{E}H_{m}(S_n)
&=\sum_{k=1}^{n}\frac{1}{m}\sum ^{m}_{j=1}\binom{m}{j}\mathbb{E}H_{m-j}\left( S_{n;k}\right)j\mathbb{E}H_j\left( X_k\right) (1-\sigma_k^2) ^{(m-j)/2}\sigma_k^{j}
\end{split}
\] 
holds for all \(m\geqslant 1\) and \(n\geqslant 2\).
\end{thm}
\begin{proof}
	Applying the identity (\ref{hermite_recurrence}) we can evaluate
\begin{equation}
\label{eq_steins_decomposition}
\begin{split}
\mathbb{E}H_{m+1}(S_n)&=\mathbb{E}\bigl(S_nH_{m}(S_n)-H_{m}'(S_n)\bigr)
\\
&=\sum_{k=1}^{n}\biggl(\mathbb{E}\sigma_kX_kH_m(\sqrt{1-\sigma_k^2}S_{n;k}+\sigma_kX_k)-\sigma_k^2\mathbb{E}H_m'(\sqrt{1-\sigma_k^2}S_{n;k}+\sigma_kX_k)\biggr)
\end{split}
	\end{equation}
for all \(m\geqslant 0\).
The identities (\ref{hermit_sum}) and (\ref{hermit_sum_diff}) applied with  \(\alpha=\sigma_k\),  \(\beta=\sqrt{1-\sigma_k^2}\), \(y=S_{n;k}\) and \(x=X_k\) take form
\[
\begin{split}
H_{m}\left( \sqrt{1-\sigma_k^2}S_{n;k}+\sigma_kX_k\right)& =\sum ^{m}_{j=0}\binom{m}{j}H_{m-j}\left( S_{n;k}\right) H_{j}\left( X_k\right) (1-\sigma_k^2) ^{(m-j)/2}\sigma_k^{j},
\\
H_{m}'\left( \sqrt{1-\sigma_k^2}S_{n;k}+\sigma_kX_k\right)&=\frac{1}{\sigma_k}\sum ^{m}_{j=0}\binom{m}{k}H_{m-j}\left(S_{n;k}\right) H_{j}'\left( X_k\right)(1-\sigma_k^2) ^{(m-j)/2}\sigma_k^{j}.
\end{split}
\]
Plugging the above two expressions into the identity (\ref{eq_steins_decomposition}) and taking into account that \(S_{n;k}\) and \(X_k\) are independent, we obtain
\[
\begin{split}
\mathbb{E}H_{m+1}(S_n)&=\sum_{k=1}^{n}\sum ^{m}_{j=0}\binom{m}{j}\mathbb{E}H_{m-j}\left( S_{n;k}\right)\mathbb{E}\biggl(X_k H_{j}\left( X_k\right) -H_{j}'\left( X_k\right)\biggr)(1-\sigma_k^2) ^{(m-j)/2}\sigma_k^{j+1}.
\end{split}
\]
Identity (\ref{hermite_recurrence}) implies that \(X_k H_{j}\left( X_k\right) -H_{j}'\left( X_k\right)=H_{j+1}\left( X_k\right)\), therefore
\[
\begin{split}
\mathbb{E}H_{m+1}(S_n)&=\sum_{k=1}^{n}\sum ^{m}_{j=0}\binom{m}{j}\mathbb{E}H_{m-j}\left( S_{n;k}\right)\mathbb{E}H_{j+1}\left( X_k\right) (1-\sigma_k^2) ^{(m-j)/2}\sigma_k^{j+1}.
\end{split}
\]
Applying the binomial identity
 \(
 \binom{m}{j}=\binom{m+1}{j+1}\frac{j+1}{m+1}
 \)
we can rewrite our expression as
\[
\begin{split}
\mathbb{E}H_{m+1}(S_n)
&=\sum_{k=1}^{n}\frac{1}{(m+1)}\sum ^{m+1}_{j=1}\binom{m+1}{j}\mathbb{E}H_{m+1-j}\left( S_{n;k}\right)j\mathbb{E}H_j\left( X_k\right) (1-\sigma_k^2) ^{(m+1-j)/2}\sigma_k^{j}.
\end{split}
\] 
Replacing here \(m\) with \(m-1\) we obtain the recurrence of the theorem.
\end{proof}
We will routinely use the notation 
\[
B(m,k,p)=\binom{m}{k}p^k(1-p)^{m-k}.
\]
\begin{lem}
\label{lem_ineq_quad_form}
	 Suppose
\[
\begin{split}
	&a_1,a_2\ldots
	\\
	&b_0,b_1,\ldots
\end{split}
\]
are two arbitrary sequences of real numbers, then 
	\[
\begin{split}
&\sum_{m=1}^\infty \left|\frac{1}{mp}\sum ^{m}_{k=1}b_{m-k}ka_k\sqrt{B\left(m,k,p\right)}\right|^2
\\
&\qquad
\leqslant
\left(\sum_{{k\in \mathbb{Z}^+\setminus J}}{p^{k-2}}\right)b_0^2\left(\sum ^{\infty}_{k=1}a_k^2\right)+C_J(p)\left(\sum ^{\infty}_{k=1}a_k^2\right)\left(\sum_{j=1 }^\infty b_{j}^2	\right),
\end{split}
\]
where \(0<p<1\) and \(J\) is a set of indices  such that \(a_j=0\) for all \(j\in J\) and
\[
C_J(p)=\max_{s\in \mathbb{Z}^+}h_J(s,p),
\]
 where \(h_J(s,p)\) is defined as a function
\begin{equation}
\label{def_h_J}
h_J(s,p)=\frac{(1-p)^s}{p^2}\sum_{k\in \mathbb{Z}^+\setminus J}\frac{k^2}{(k+s)^2}\binom{k+s}{k}p^k.
\end{equation}
Here and throughout, \(\mathbb{Z}^+\) denotes the set of all positive integers.
\end{lem}

\begin{proof}
Let us denote
\begin{equation}
\label{def_c_m}
c_m
=\frac{1}{mp}\sum ^{m}_{k=1}b_{m-k}ka_k\sqrt{B\left(m,k,p\right)}
\end{equation}
for all \(m\geqslant 1\).
	Applying Cauchy inequality we get
\[
\begin{split}
\bigl|c_m\bigr|^2
&=\frac{1}{m^2p^2}\left(\sum ^{m}_{k=1}b_{m-k}ka_k\sqrt{B\left(m,k,p\right)}\right)^2
\\
&\leqslant\frac{1}{m^2p^2}\left(\sum ^{m}_{k=1}a_k^2\right)\left(\sum_{\substack{1\leqslant k\leqslant m\\k\not\in J}}\bigl|b_{m-k}\bigr|^2k^2B\left(m,k,p\right)	\right)
\\
&\leqslant\left(\sum ^{\infty}_{k=1}a_k^2\right)\frac{1}{m^2p^2}\left(\sum_{\substack{1\leqslant k\leqslant m\\k\not\in J}}\bigl|b_{m-k}\bigr|^2k^2B\left(m,k,p\right)	\right).
\end{split}
\]
Summing this inequality from \(m=1\) to \(\infty\) we obtain
\[
\begin{split}
\sum_{m=1}^\infty \bigl|c_m\bigr|^2&\leqslant\left(\sum_{k=1}^\infty  a_k^2\right)\sum_{m=1}^\infty\frac{1}{m^2p^2}\left(\sum_{\substack{1\leqslant k\leqslant m\\k\not\in J}}\bigl|b_{m-k}\bigr|^2k^2B\left(m,k,p\right)	\right)
\\
&\leqslant
\left(\sum ^{\infty}_{k=1}a_k^2\right)\sum ^{\infty}_{\substack{k=1\\k\not\in J}}\sum_{m=k}^\infty \bigl|b_{m-k}\bigr|^2\frac{k^2B\left(m,k,p\right)}{m^2p^2}	
\\
&\leqslant
\left(\sum ^{\infty}_{k=1}a_k^2\right)\sum ^{\infty}_{\substack{k=1\\k\not\in J}}\sum_{s=0}^\infty \bigl|b_{s}\bigr|^2\frac{k^2B\left(k+s,k,p\right)}{(k+s)^2p^2}	
\\
&\leqslant
\left(\sum ^{\infty}_{k=1}a_k^2\right)\sum_{s=0}^\infty \bigl|b_{s}\bigr|^2\sum ^{\infty}_{\substack{k=1\\k\not\in J}}\frac{k^2B\left(k+s,k,p\right)}{(k+s)^2p^2}	
\end{split}
\]
thus
\[
\begin{split}
\sum_{m=1}^\infty \bigl|c_m\bigr|^2
&\leqslant
\left(\sum ^{\infty}_{k=1}a_k^2\right)\sum_{s=0}^\infty \bigl|b_{s}^{(n-1)}\bigr|^2h_J(s,p)	
\end{split}
\]
where
\[
h_J(s,p)=\sum _{{k\in \mathbb{Z}^+\setminus J}}\frac{k^2B\left(k+s,k,p\right)}{(k+s)^2p^2}=\frac{(1-p)^s}{p^2}\sum_{{k\in \mathbb{Z}^+\setminus J}}\frac{k^2}{(k+s)^2}\binom{k+s}{k}p^k
\]
note that
\[
h_J(0,p)=\sum _{k\in \mathbb{Z}^+\setminus J}\frac{B\left(k,k,p\right)}{p^2}=\sum _{k\in \mathbb{Z}^+\setminus J}\frac{p^k}{p^2}
\]
therefore
\[
\begin{split}
\sum_{m=1}^\infty \bigl|c_m\bigr|^2
&\leqslant
\left(\sum ^{\infty}_{k=1}a_k^2\right)\sum _{k\in \mathbb{Z}^+\setminus J}{p^{k-2}}+\left(\sum ^{\infty}_{k=1}a_k^2\right)\sum_{s=1}^\infty \bigl|b_{s}\bigr|^2h_J(s,p).
\end{split}
\]
Evaluating \(h_J(s,p)\) with its upper bound \(C_J(p)\) on the right side of the above inequality and recalling the definition (\ref{def_c_m}) of \(c_m\) we complete the proof of the lemma.
\end{proof}

\begin{lem}
\label{lem_C_estimates} The following estimates for \( C_J(p) \) hold for \( 0 < p < 1 \), where \( J \) is a specific set defined below:
\begin{enumerate}
	\item If \( J = \{1,2\} \), then
	\[
	C_{\{1,2\}}(p) < 1.2183 + 1.6066 \frac{p}{1-p}.
	\]
	\item For the set \( J_{sym} = \{2\} \cup \{2m+1 \mid m \geqslant 0\} \), which includes the number \( 2 \) as well as all odd positive integers, the following inequality holds:
	\[
	C_{J_{sym}}(p) < 0.5893 + 0.9724 \frac{p}{1-p} + 0.1405 \frac{p^2}{(1-p^2)^2}.
	\]
\end{enumerate}
Upper bounds for the first nine values of \( C_{\{1,2\}}(1/n) \) and \( C_{J_{sym}}(1/n) \) are provided in Table \ref{tab:rounded_estimates}.
 \begin{table}[h]
    \centering
    \begin{tabular}{|c|c|c|c|c|c|c|c|c|c|}
        \hline
        $n$ & $2$ & $3$ & $4$ & $5$ & $6$ & $7$ & $8$ & $9$ & $10$ \\
        \hline
        $C_{\{1,2\}}(1/n)$  & 2.1327  & 1.6582  & 1.5043  & 1.4293  & 1.3851  & 1.3560  & 1.3354  & 1.3202  & 1.3085  \\
        $C_{J_{sym}}(1/n)$ & 1.0570  & 0.8168  & 0.7387  & 0.7001  & 0.6773  & 0.6622  & 0.6515  & 0.6436  & 0.6374  \\
        \hline
    \end{tabular}
    \caption{The values in the table represent upper bounds for $C_{\{1,2\}}(1/n)$ and $C_{\text{sym}}(1/n)$, meaning the actual values do not exceed these estimates.}
    \label{tab:rounded_estimates}
\end{table}
\end{lem}
The proof of this lemma consisting of a tedious series of elementary inequalities is provided in the Appendix \ref{appA}.
\begin{thm}
\label{thm_general_sigma} Let
\[
S_n=\sigma_1X_1+\sigma_2X_2+\cdots+\sigma_nX_n
\]
where \(X_j\) are independent, \(\mathbb{E}X_j=0\), \(\mathbb{E}X_j^2=1\)  and \(0< \sigma_j^2<1\) for all \(1\leqslant j \leqslant n\). Assume further that the variances satisfy  \(\sigma_1^2+\sigma_2^2+\cdots+\sigma_n^2=1\). Let us also assume that  \(J\) is a set of integers such that 
\[
\mathbb{E}H_k(X_j)=0
\]
whenever \(k\in J\) for all \(j=1,2,\ldots,n\). 
Let us also define
\[
S_{n;k}=\frac{S_n-\sigma_kX_k}{\sqrt{1-\sigma_k^2}}.
\]
Then 
	\[
\chi^2(S_n,\mathcal{N})\leqslant\sum_{k=1}^{n}\left(\sum _{{j\in \mathbb{Z}^+\setminus J}}{\sigma^{2(j-1)}_k}\right)\chi^2(X_k,\mathcal{N})+\sum_{k=1}^{n}\sigma_k^2C_J(\sigma_k^2)\chi^2(X_k,\mathcal{N}) \chi^2(S_{n;k},\mathcal{N})
\]
for all \(n\geqslant 2\).

\noindent
\textbf{Remark.} Since each \( X_j \) has zero mean and unit variance, it follows that \( \mathbb{E}H_1(X_j) = \mathbb{E}H_2(X_j) = 0 \) for all \( j \). Hence, the integers \( 1 \) and \( 2 \) can always be assumed to belong to the set \( J \).
\end{thm}
\begin{proof}
Dividing both sides of the recurrence identity in Theorem~\ref{stein_recurrent_identity} by \( \sqrt{m!} \), we can rewrite it in the following form:
\[
\begin{split}
\frac{\mathbb{E}H_{m}(S_n)}{\sqrt{m!}}
&=\frac{1}{m}\sum_{k=1}^{n}\sum ^{m}_{j=1}\frac{\mathbb{E}H_{m-j}\left( S_{n;k}\right)}{\sqrt{(m-j)!}}j\frac {\mathbb{E}H_{j}\left( X_k\right)}{\sqrt{j!}}\sqrt{B(m,j,\sigma_k^2)}.
\end{split}
\] 
We introduce the following notation for simplicity
\[
b_m^{(n)}=\frac{\mathbb{E}H_{m}(S_n)}{\sqrt{m!}},\quad b_m^{(n;k)}=\frac{\mathbb{E}H_{m}(S_{n;k})}{\sqrt{m!}},\quad a_{j}^{(k)}=\frac {\mathbb{E}H_{j}\left( X_k\right)}{\sqrt{j!}}
\]
with these notations, the recurrence relation takes the following form:
\[
\begin{split}
b_m^{(n)}
&=\frac{1}{m}\sum_{k=1}^{n}\sum ^{m}_{j=1}b_{m-j}^{(n;k)}ja_{j}^{(k)}\sqrt{B(m,j,\sigma_k^2)}.
\end{split}
\]
Applying here the Cauchy inequality, we obtain
\[
\begin{split}
\bigl|b_m^{(n)}\bigr|^2
&=\left(\sum_{k=1}^{n}\sigma_k\frac{1}{m\sigma_k}\sum ^{m}_{j=1}b_{m-j}^{(n;k)}ja_{j}^{(k)}\sqrt{B(m,j,\sigma_k^2)}\right)^2
\\
&\leqslant\left(\sum_{k=1}^{n}\sigma_k^2\right)\sum_{k=1}^{n}\left(\frac{1}{m\sigma_k}\sum ^{m}_{j=1}b_{m-j}^{(n;k)}ja_{j}^{(k)}\sqrt{B(m,j,\sigma_k^2)}\right)^2
\\
&\leqslant\sum_{k=1}^{n}\left(\frac{1}{m\sigma_k}\sum ^{m}_{j=1}b_{m-j}^{(n;k)}ja_{j}^{(k)}\sqrt{B(m,j,\sigma_k^2)}\right)^2
\end{split}
\]
here we have used the fact that \(\sigma_1^2+\sigma_2^2+\cdots+\sigma_n^2=1\). Summing both sides of the above inequality over \( m = 1, 2, \ldots \) and applying the estimate from Lemma~\ref{lem_ineq_quad_form}, we obtain:
\[
\begin{split}
\sum_{m=1}^\infty\bigl|b_m^{(n)}\bigr|^2
&\leqslant\sum_{k=1}^{n}\sigma_k^2\sum_{m=1}^\infty\left(\frac{1}{m\sigma_k^2}\sum ^{m}_{j=1}b_{m-j}^{(n;k)}ja_{j}^{(k)}\sqrt{B(m,j,\sigma_k^2)}\right)^2
\\
&\leqslant\sum_{k=1}^{n}\sigma_k^2\left(\left(\sum _{\substack{j\geqslant 1\\j\not\in J}}{\sigma^{2(j-2)}_k}\right)\left(\sum ^{\infty}_{j=1}|a_j^{(k)}|^2\right)+C_J(\sigma_k^2)\left(\sum ^{\infty}_{j=1}|a_j^{(k)}|^2\right)\sum_{s=1 }^\infty \bigl|b_{s}^{(n;k)}\bigr|^2\right).
\end{split}
\]
Hence, noting that \( \chi^2(S_n, \mathcal{N}) = \sum_{m=1}^\infty \bigl|b_m^{(n)}\bigr|^2 \), \( \chi^2(S_{n;k}, \mathcal{N}) = \sum_{m=1}^\infty \bigl|b_m^{(n;k)}\bigr|^2 \), and \( \chi^2(X_k, \mathcal{N}) = \sum_{m=1}^\infty \bigl|a_m^{(k)}\bigr|^2 \), the inequality stated in the theorem follows immediately.
\end{proof}
\subsection{Sums of random variables with equal variances}
By iterating the upper bound provided by Theorem \ref{thm_general_sigma}, we obtain an upper bound for \(\chi^2(S_n, \mathcal{N})\) in the form of an \( n \)-th order polynomial in the variables \(\chi^2(X_j, \mathcal{N})\), whose coefficients depend on \(\sigma_1^2, \sigma_2^2, \dots, \sigma_n^2\). However, since analyzing such polynomials is quite intricate, we will henceforth restrict our attention to the simplest case, where all summands of \( S_n \) have the same variance, equal to \( 1/n \), i.e.,  
\[
\sigma_1 = \sigma_2 = \cdots = \sigma_n = \frac{1}{\sqrt{n}}.
\]
In this case, the theorem takes the following form.
\begin{thm}
\label{thm_equal_variances}
Suppose 
\[
S_n=\frac{X_1+X_2+\cdots+X_n}{\sqrt{n}}
\]
where \(X_j\) are independent, \(\mathbb{E}X_j=0\) and \(\mathbb{E}X_j^2=1\). Let us also define
\[
S_{n;k}=\frac1{\sqrt{n-1}}{\sum_{\substack{1\leqslant j\leqslant n\\j\not=k}}X_j}
\]
then
\[
\chi^2(S_n,\mathcal{N})\leqslant \left(\sum _{\substack{j\in \mathbb{Z}^+\setminus J}}\frac1{n^{j-2}}\right)\frac1n\sum_{k=1}^{n}\chi^2(X_k,\mathcal{N})+\frac{C_J(1/n) }{n} \sum_{k=1}^{n} \chi^2(X_k,\mathcal{N})\chi^2(S_{n;k},\mathcal{N})
\]
where \(J\) is a set of integers such that 
\[
\mathbb{E}H_k(X_j)=0
\]
whenever \(k\in J\) for all \(j=1,2,\ldots,n\).
\end{thm}

\begin{cor}
Under the conditions of the previous theorem 
\begin{equation}
\label{rec_chi_mean_0_var_1}
\chi^2(S_n,\mathcal{N})\leqslant \frac{1}{n(n-1)}\sum_{k=1}^{n}\chi^2(X_k,\mathcal{N})+\frac{C_{\{1,2\}}(1/n)}{n} \sum_{k=1}^{n} \chi^2(X_k,\mathcal{N})\chi^2(S_{n;k},\mathcal{N}).	
\end{equation}
If we additionally  assume that all \(X_j\) are symmetric  (that is the distribution of \(-X_j\) is the same as that of \(X_j\)) then  we get the inequality
\begin{equation}
\label{rec_chi_sym}
\chi^2(S_n,\mathcal{N})\leqslant \frac{1}{n(n^2-1)}\sum_{k=1}^{n}\chi^2(X_k,\mathcal{N})+\frac{C_{J_{sym}}(1/n)}{n} \sum_{k=1}^{n} \chi^2(X_k,\mathcal{N})\chi^2(S_{n;k},\mathcal{N}).
\end{equation}
where  \(J_{sym}=\{2\}\cup \{2m+1|m\geqslant 0\}\).
\end{cor}
\begin{proof}Note that requirement of the  Theorem \ref{thm_equal_variances}  that all variables \(X_j\) have zero mean and variances equal to one implies that \(\mathbb{E}H_{1}(X_j)=\mathbb{E}H_{2}(X_j)=0\). This means that  we can apply this theorem with  \(J=\{1,2\}\). The sum over \(\mathbb{Z}^+\setminus J\) on the right side of the resulting inequality can be easily computed
\[
\sum_{j\in \mathbb{Z}^+\setminus\{1,2\}}\frac{1}{n^{j-2}}=\sum_{m=1}^\infty\frac{1}{n^m}=\frac{1}{n-1}.
\]
In the case when all \(X_j\) are symmetric, all odd moments of these variables will be equal to zero, and as a consequence all expectations of Hermite polynomials of odd order will also be equal \(\mathbb{E}H_{2m+1}(X_j)=0\) for all \(m\geqslant 0\) and \(1\leqslant j \leqslant n\). This means that we can apply the inequality of the Theorem \ref{thm_equal_variances} with  \(J=J_{sym}\). In this case
\[
\sum_{j\in \mathbb{Z}^+\setminus J_{sym}}\frac{1}{n^{j-2}}=\frac{1}{n^2}+\frac{1}{n^4}+\frac{1}{n^6}\cdots=\frac{1}{n^2-1}.
\]
The corollary is proved.
	\end{proof}
We will use a special case of Maclaurin's inequality, stated in the following theorem (see \cite{hardy_littlewood_polya_1952}, p. 52, or \cite{steele_2004}, p. 179).
\begin{thm}[Maclaurin's inequality]For any sequence of real non-negative numbers \(a_1,a_2,\ldots, a_n\) holds the inequality
\begin{equation}
\label{Maclaurins_inequality}
\begin{split}
\frac{1}{n_{(k)}} \sum_{ \substack{1\leqslant j_1,j_2,\ldots,j_k\leqslant n\\
j_1\not=j_2\not=\cdots\not=j_k} }a_{j_1} a_{j_2}\cdots a_{j_k}\leqslant
\left(\dfrac{1}{n}\sum ^{n}_{m=1}a_{m}\right)^k
\end{split}	
\end{equation}
where we have used the notation \(n_{(k)}=n(n-1)(n-2)\cdots(n-k+1)\).
\end{thm}

\begin{prop}
\label{recurrencies_main}
Suppose we have a function \( \mu: 2^{\Omega_n} \to [0, +\infty) \) defined on all subsets of a finite set \( \Omega_n = \{1,2,3,\dots,n\} \) that satisfies the inequality  
\[
\mu(A) \leqslant \frac{1}{|A|} \sum_{k \in A} \mu(\{k\}) + \frac{C_{|A|}}{|A|} \sum_{k \in A} \mu(\{k\}) \mu(A \setminus \{k\})
\]
for all subsets \( A \subset \Omega_n \) with at least two elements \(|A| \geqslant 2\), where \( C_2, C_3, \dots, C_n \) is a sequence of positive numbers.
Then for all sets \(A\subset \Omega_n\) holds the inequality
\[
\begin{split}
\mu(|A|)&\leqslant \frac{1}{|A|}\sum_{k\in A }\mu(\{k\})+\sum_{k=2}^{|A|}{C_{|A|} C_{|A|-1}\cdots C_{|A|-k+2}}\left(\frac{1}{|A|}\sum_{k\in A}\mu(\{k\})\right)^k.
\end{split}
\]

\end{prop}
\begin{proof} Applying induction on the size of set \(A\) starting with \(|A|=2\)
we can easily prove that
\[
\begin{split}
\mu(A)&\leqslant \frac{1}{|A|}\sum_{k\in A}\mu(\{k\})
+\sum_{k=2}^{|A|}{C_{|A|} C_{|A|-1}\cdots C_{|A|-k+2}}\left(\frac{1}{|A|_{(k)}} \sum_{ \substack{j_1\in A,j_2\in A,\ldots,j_k\in A \\
j_1\not=j_2\not=\cdots\not=j_k} }\mu(\{j_1\})\mu(\{j_2\}) \cdots \mu(\{j_k\})\right).
\end{split}
\]
Applying Maclaurin's inequality~\eqref{Maclaurins_inequality} to evaluate the symmetric means on the right-hand side of this identity completes the proof of the proposition.
\end{proof}
\begin{thm}
\label{main_theorem}
Suppose \(X_1,X_2,\ldots,X_n\) are independent random variables such that \(\mathbb{E}X_j=0\) and \(\mathbb{E}X_j^2=1\)  for all \(j=1,2,\ldots,n\). Then holds the inequality
\begin{equation}
\label{general_inequality}
\begin{split}
\chi^2(S_n,\mathcal{N})&\leqslant \frac1{n-1}\left(\frac{1}{n}\sum_{j=1}^{n}\chi^2(X_{j},\mathcal{N})\right)\left(1+
\sum_{k=2}^{n}{D_n D_{n-1}\cdots D_{n-k+2}} \left(\frac{1}{n}\sum_{j=1}^{n}\chi^2(X_{j},\mathcal{N})\right)^{k-1}\right)
\end{split}		
	\end{equation}
where the sequence \(D_2,D_3,\ldots\) is a sequence of positive absolute constants  such that
\[
\lim_{n\to \infty}D_n=\lim_{p\to 0}C_{\{1,2\}}(p)< 1.2183.
\]
If  in addition to previous conditions we know that all \(X_1,X_2,\ldots,X_n\) are symmetric random variables (that is the distribution of \(-X_j\) is the same as that of \(X_j\)), then
\begin{equation}
\label{general_inequality_symmetric}
\begin{split}
\chi^2(S_n,\mathcal{N})&\leqslant \frac1{n^2-1}\left(\frac{1}{n}\sum_{j=1}^{n}\chi^2(X_{j_1},\mathcal{N})\right)\left(1+
\sum_{k=2}^{n}{L_n L_{n-1}\cdots L_{n-k+2}} \left(\frac{1}{n}\sum_{j=1}^{n}\chi^2(X_{j_1},\mathcal{N})\right)^{k-1}\right)
\end{split}	
\end{equation}
where the sequence \(L_2,L_3,\ldots\) is a sequence of positive absolute constants  such that
\[
\lim_{n\to \infty}L_n=\lim_{p\to 0}C_{J_{sym}}(p)< 0.5893.
\]
\end{thm}

\begin{proof}
Multiplying both sides of inequality (\ref{rec_chi_mean_0_var_1}) by \(n-1\) we get
	\[
	\chi^2(S_n,\mathcal{N})(n-1)\leqslant \frac{1}{n}\sum_{k=1}^{n}\chi^2(X_k,\mathcal{N})+\frac{C_{\{1,2\}}(1/n) \frac{n-1}{n-2}}{n} \sum_{k=1}^{n} \chi^2(X_k,\mathcal{N})(n-2)\chi^2(S_{n;k},\mathcal{N})	\]
	for all \(n\geqslant 3\). For \(n=2\) we have
\[
	\chi^2(S_2,\mathcal{N})\leqslant \frac{1}{2}\sum_{k=1}^{2}\chi^2(X_k,\mathcal{N})+\frac{C_{\{1,2\}}(1/2) }{2} \sum_{k=1}^{2} \chi^2(X_k,\mathcal{N})\chi^2(S_{2;k},\mathcal{N})	\]	
Let us define
\[
D_n=\begin{cases}
	C_{\{1,2\}}(1/n)\frac{n-1}{n-2}&\text{if    } n\geqslant 3
	\\
	C_{\{1,2\}}(1/n)&\text{if }n=2
\end{cases}
\]
and
\[
\nu(A) =
\begin{cases}
    (|A| - 1) \, \chi^2\left(\frac{1}{\sqrt{|A|}} \sum_{j \in A} X_j, \mathcal{N} \right), & \text{if } |A| \geqslant 2, \\
    \chi^2\left(X_j, \mathcal{N} \right), & \text{if } |A| = 1 \text{ with } A = \{j\}.
\end{cases}
\]
With these notations the inequality (\ref{rec_chi_mean_0_var_1}) will take form \[
\nu(A)\leqslant\frac{1}{|A|}\sum_{k\in A}\nu(\{k\})+\frac{C_{|A|}}{|A|}\sum_{k\in A}\nu(\{k\})\nu(A\setminus\{k\}).
\]
Application of Propositon (\ref{recurrencies_main})  with \(\mu=\nu\), \(C_n=D_n\) and \(A=\{1,2,\ldots,n\}\) immediately leads to the estimate (\ref{general_inequality}) of the theorem.

For the symmetric case  let us define
\begin{equation}
\label{def_L}
L_n=\begin{cases}
	C_{J_{sym}}(1/n)\frac{n^2-1}{(n-1)^2-1}&\text{if    } n\geqslant 3
	\\
	C_{J_{sym}}(1/n)(n^2-1)&\text{if }n=2
\end{cases}
\end{equation}
and
\[
\omega(A) =
\begin{cases}
    (|A|^2 - 1) \, \chi^2\left(\frac{1}{\sqrt{|A|}} \sum_{j \in A} X_j, \mathcal{N} \right), & \text{if } |A| \geqslant 2, \\
    \chi^2\left(X_j, \mathcal{N} \right), & \text{if } |A| = 1 \text{ with } A = \{j\}.
\end{cases}
\]
With these notations the inequality (\ref{rec_chi_sym}) will take form
\[
\omega(A)\leqslant\frac{1}{|A|}\sum_{k\in A}\omega(\{k\})+\frac{L_{|A|}}{|A|}\sum_{k\in A}\omega(\{k\})\omega(A\setminus\{k\}).
\]
Applying again Proposition (\ref{recurrencies_main}) with \(\mu=\omega\), \(C_n=L_n\) and \(A=\{1,2,\ldots,n\}\) we obtain the estimate (\ref{general_inequality_symmetric}) of the theorem.
\end{proof}
\begin{exm}
	
Let us consider the case of a sum of \(n\) identically distributed independent random variables, each of which has a uniform distribution on a finite interval
 with zero mean and variance equal to one. The density of such random variables is  
\[
p(x) =
\begin{cases}
    \frac{1}{2\sqrt{3}}, & \text{if } |x| \leqslant \sqrt{3}, \\
    0, & \text{otherwise}.
\end{cases}
\]
The \(\chi^2\) distance from normal distribution of a random variable \(Y\) with such a density is
\[
\chi^2(X,\mathcal{N})=\int ^{+\infty }_{-\infty }\dfrac{p^2\left( x\right) }{\varphi \left( x\right) }\, dx-1=\int ^{\sqrt{3} }_{-\sqrt{3} }\dfrac{\left(\frac{1}{2\sqrt{3}}\right)^2 }{\frac{e^{-x^2/2}}{\sqrt{2\pi}}}\, dx-1
=0.3285\ldots
\]
Thus in the case when all random variables \(X_j\) are uniformly distributed on the interval \([-\sqrt{3},\sqrt{3}]\) the inequality (\ref{general_inequality_symmetric}) will take form
\begin{equation}
\label{ineq_uniform_dist}
\begin{split}
\chi^2(S_n,\mathcal{N})&\leqslant \frac{\chi^2(X_{1},\mathcal{N})}{n^2-1}\left(1+
\sum_{k=2}^{n}{L_n L_{n-1}\cdots L_{n-k+2}} \bigl(\chi^2(X_{1},\mathcal{N})\bigr)^{k-1}\right)
\end{split}		
\end{equation}

Note that the estimates of Lemma \ref{lem_C_estimates} together with the definition (\ref{def_L}) of \(L_n\) imply that inequality
\(
L_n< 2.18
\)
will be true for all \(n\geqslant 3\) and \(L_2<3.2\).
Thus the upper bound for the product
\[
L_n L_{n-1}\cdots L_{n-k+2}\leqslant \frac{3.2}{2.18}2.18^{k-1}
\]
will be true for all \(n\geqslant k\geqslant 2\).  Applying this estimate together with  \(\chi^2(X_1,\mathcal{N})<0.33 \) to evaluate the  terms under the summation sign inside the inequality (\ref{ineq_uniform_dist}) we get the upper bound
	\[
\begin{split}
\chi^2(S_n,\mathcal{N})
&<\frac{1.6}{n^2-1},
\end{split}
	\]	
which is valid for all \(n\geqslant 2\).
\end{exm}
\begin{proof}[Proof of Theorem \ref{main_corollary}]
According to the conditions of the theorem  the inequality
\[
D_j\leqslant 1.218
\]
holds for all  \(j\) with exception of a finite number of indices. This means that  the product \(D_n D_{n-1}\cdots D_{n-k+2}\) can only contain  a finite   (and bounded by an absolute constant) number of terms  larger  than \(1.218\). Therefore
\[
D_n D_{n-1}\cdots D_{n-k+2}\leqslant D(1.218)^{k-1}
\]
for all \(2\leqslant k\leqslant n\) where 
\[
D=\prod_{\substack{j\geqslant 2\\
D_j>1.218}}\frac{D_j}{1.218}.
\]
Hence 
\[
\begin{split}
\sum_{k=2}^{n}{D_n D_{n-1}\cdots D_{n-k+2}} \left(\frac{1}{n}\sum_{j=1}^{n}\chi^2(X_{j},\mathcal{N})\right)^{k-1}
&\leqslant D\left(\frac{1}{n}\sum_{j=1}^{n}\chi^2(X_{j},\mathcal{N})\right)\sum_{k=2}^{n} \left(1.218*0.82\right)^{k-2}
\\
&\leqslant D\left(\frac{1}{n}\sum_{j=1}^{n}\chi^2(X_{j},\mathcal{N})\right)\sum_{k=2}^{\infty} \left(0.99876\right)^{k-2}.
\end{split}
\]
Hence immediately follows the first inequality of the theorem. The proof of the  inequality for symmetric random variables is identical.
\end{proof}
\subsection{Relation between subgausian condition and  \(\chi^2\)-distance}

\begin{thm}
\label{thm_subgausian_general}
	Suppose \(Y\) is a random variable with a finite \(\chi^2(Y,\mathcal{N})\) and \(J\) is a set of positive integers  such that 
	\[
	\mathbb{E}H_n(Y)=0
	\]
	for all \(n\in J\). 
	If
	\[
	\chi^2(Y,\mathcal{N})<\inf_{x> 0}\frac{(e^{x/2}-1)^2}{e^{x}-1-\sum_{\substack{n\in J}}\frac{x^{n}}{n!}}
	\]
	then the condition 
\[
\mathbb{E}e^{tY}<e^{t^2}
\]
holds for all \(t\not=0 \). 

\noindent
\textbf{Remark.} Note that for the infimum to be non-zero one necessarily needs that \(1\in J\), or in other words \(\mathbb{E}H_1(Y)=\mathbb{E}Y=0\).
\end{thm}
\begin{proof}Substituting \( x = Y \) into the generating function of Hermite polynomials,  
\[
e^{xt - t^2/2} = \sum_{n=0}^{\infty} \frac{H_n(x)}{n!} t^n,
\]
and taking the expectation on both sides, we obtain the identity  
\begin{equation}
\label{identity_for_H}
\mathbb{E} e^{tY} = e^{t^2/2} \sum_{n=0}^{\infty} \frac{\mathbb{E} H_n(Y)}{n!} t^n.	
\end{equation}Applying Cauchy inequality to evaluate the right side of this identity  and keeping in mind that \(\mathbb{E}H_n(Y)=0\) for all \(n\in J\) we get
\[
\begin{split}
\mathbb{E}e^{tY}&=e^{t^2/2}\sum ^{\infty }_{n=0}\frac{\mathbb{E}H_n(Y)}{n!}t^n	
\\
&\leqslant e^{t^2/2}\left(1+\sqrt{\sum ^{\infty }_{n=1}\frac{\bigl|\mathbb{E}H_n(Y)\bigr|^2}{n!}}\sqrt{\sum_{\substack{n\geqslant 1\\n\not\in J}} \frac{t^{2n}}{n!}}\right)
\\
&=e^{t^2/2}\left(1+(e^{t^2/2}-1)\sqrt{\chi^2(Y,\mathcal{N})\frac{e^{t^2}-1-\sum_{\substack{n\in J}} \frac{t^{2n}}{n!}}{(e^{t^2/2}-1)^2}}\right).
\end{split}
\]
If the condition of the theorem is satisfied, then
\[
\chi^2(Y,\mathcal{N})<\frac{(e^{t^2/2}-1)^2}{e^{t^2}-1-\sum_{\substack{n\in J}} \frac{t^{2n}}{n!}}
\]
for all \(t\not=0\), 
which leads to inequality
\[
\begin{split}
\mathbb{E}e^{tY}
&<e^{t^2/2}\left(1+(e^{t^2/2}-1)\right)
=e^{t^2}
\end{split}
\] 
for all \(t\not=0\). The theorem is proved.
\end{proof}
\begin{cor}
\label{thm_subgausian_intro}
Below are  three  sufficient conditions under which the random variable \(Y\)  satisfies the subgaussian  condition
\(
\mathbb{E}e^{tY}<e^{t^2}
\)
for all \(t\not=0 \).
\begin{enumerate}
\item
If \(\mathbb{E}Y=0\) and
\[
\chi^2(Y,\mathcal{N})<\inf_{x> 0}\frac{(e^{x/2}-1)^2}{e^{x}-1-x}=\frac{1}{2}
\]
	\item  If \(\mathbb{E}Y=0\), \(\mathbb{E}Y^2=1\) and
\[
\chi^2(Y,\mathcal{N})<\min_{x> 0}\frac{(e^{x/2}-1)^2}{e^{x}-1-x-\frac{x^2}{2!}}=0.96116\ldots
\]	
\item
If  \(Y\) is a symmetric random variable with zero mean and unit variance, and
\[
\chi^2(Y,\mathcal{N})<\min_{x> 0}\frac{(e^{x/2}-1)^2}{\frac{e^{x}+e^{-x}}{2}-1-\frac{x^2}{2!}}=1.97044\ldots
\]
\end{enumerate}
\end{cor}

\begin{proof}
The proof follows immediately from Theorem \ref{thm_subgausian_general} by considering special cases \(J=\{1\}\), \(J=\{1,2\}\) and  \(J=J_{sym}\) .
\end{proof}
The upper bound constants in the above theorem are unlikely to be optimal. Proposition 18.1 in \cite{bobkov_et_al_2019} presents an example of a distribution \(Y\) with zero mean and unit variance for which \( \chi^2(Y, \mathcal{N}) \) is finite, yet the subgaussian condition~\eqref{sub-gausian_condition} fails to hold.

\appendix

\section{Appendix: Upper bounds for \(C_J(p)\)}
\label{appA}

Instead of dealing with \(h_J(s,p)\) we will consider a simpler upper bound estimate \(h^0_J(s,p)\) for this function which is more convenient to deal with, without sacrificing too much in terms of precision. 
\begin{lem}
\label{lem_h0_vs_h} The inequality
	\[
h_J(s,p)\leqslant h^0_J(s,p)
\]
with
\[
\begin{split}
h_J^0(s,p)=\frac{(1-p)^s}{p^2}\left(\frac{p^2}{(1-p)^{s+1}}+\frac{p}{s}\frac{1}{(1-p)^{s}}-\sum_{k\in  J}k\frac{(k+s-2)!}{(k-1)!s!}p^k\right)
\end{split}
\]
holds for any set \(J\subset \mathbb{Z}^+\) and all \(s\geqslant 1\) and \(0<p<1\). 
\end{lem}
\begin{proof}
	We express the binomial coefficient \(\binom{k+s}{k}\) in the definition (\ref{def_h_J}) of \( h_J(s,p) \) as a ratio of factorials
\(
\binom{k+s}{k} = \frac{(k+s)!}{k!s!},
\)  
 since \(\frac{k+s-1}{k+s} < 1\), it follows that
\begin{equation}
\label{h_j_ineq}
\begin{split}
	h_J(s,p)
	&\leqslant\frac{(1-p)^s}{p^2}\sum_{k\in \mathbb{Z}^+\setminus J}k\frac{(k+s-2)!}{(k-1)!s!}p^k
	\\
	&=\frac{(1-p)^s}{p^2}\left(\sum_{k=0}^\infty k\frac{(k+s-2)!}{(k-1)!s!}p^k-\sum_{k\in  J}k\frac{(k+s-2)!}{(k-1)!s!}p^k\right).
\end{split}
\end{equation}
Making use of identity
	\(
\sum_{k=0}^\infty\frac{(k+s)!}{k!s!}p^k=\frac{1}{(1-p)^{s+1}}
\)
we can compute the series 
\[
\begin{split}
\sum ^{\infty}_{k=1}k\frac{(k+s-2)!}{(k-1)!s!}p^k
&=\frac{p^2}{(1-p)^{s+1}}+\frac{p}{s}\frac{1}{(1-p)^{s}}
\end{split}.
\]
Inserting this expression into the right side of the formula (\ref{h_j_ineq}) we complete the proof of the lemma.
\end{proof} 
As a special case when \(J=\{1,2\}\) we get
\begin{equation}
\label{h_general_formula}
	\begin{split}
h^0_{\{1,2\}}(s,p)
&=\frac{1}{1-p}-2(1-p)^s+\frac{1-(1-p)^s}{ps }.	
\end{split}
\end{equation}
For the set \(J_{sym}=\{2\}\cup \{2m+1|m\geqslant 0\}\) note that the Taylor expansion of \( \frac{(1-p)^s}{p^2} h^0_{J_{sym}}(s,p) \) can be obtained from the Taylor expansion of \( \frac{(1-p)^s}{p^2} h^0_{J_{\{1,2\}}}(s,p) \) by removing all terms corresponding to odd powers of \( p \). Hence
\[
\frac{p^2}{(1-p)^s}h^0_{J_{sym}}(s,p)=\frac{1}{2}\left( \frac{p^2}{(1-p)^s}h_{\{1,2\}}(s,p)+\frac{p^2}{(1+p)^s}h_{\{1,2\}}(s,-p)\right)
\]
hence
\begin{equation}
\label{sym_h_formula}
	h^0_{J_{sym}}(s,p)=\frac{1}{2}\left( h^0_{\{1,2\}}(s,p)+\frac{(1-p)^s}{(1+p)^s}h^0_{\{1,2\}}(s,-p)\right).
\end{equation}
We will need several elementary estimates collected in the following lemma.
\begin{lem}
\label{inequality_elementary} Inequalities
\begin{equation}
\label{ineq_first}
\begin{split}
	0<e^{sp}-\left( 1+p\right) ^{s}<\frac{sp^2}{2}e^{sp},
\end{split}	
\end{equation}
\begin{equation}
\label{ineq_second}
\begin{split}
	0<e^{-sp}-\left( 1-p\right) ^{s}<\frac{sp^2}{2(1-p)}e^{-sp}
\end{split}	
\end{equation}
and 
\begin{equation}
\label{ineq_third}
\begin{split}
0<e^{-2sp}-\frac{(1-p)^s}{(1+p)^s}&< \frac{2}{3}\frac{sp^3}{(1-p^2)^2}e^{-2sp}
\end{split}	
\end{equation}
hold for all \(0<p<1\) and \(s\geqslant 1\).
\end{lem}
\begin{proof}

The lower bound of the first inequality follows from the well-known elementary inequality \( 1 + p < e^p \). For the upper bound, note that this inequality also implies \( (1 + p)e^{-p} < 1 \). Therefore, 
\[
\begin{split}
	e^{sp}-\left( 1+p\right) ^{s}&=se^{sp}\int_0^p(\left( 1+x\right) e^{-x})^{s-1 }xe^{-x}\,dx
	<se^{sp}\int_0^pxe^{-x}\,dx<\frac{sp^2}{2}e^{sp}.
\end{split}
\]
In a similar way 
\[
\begin{split}
	e^{-sp}-\left( 1-p\right) ^{s}
	&=se^{-sp}\int_0^p(\left( 1-x\right) e^{x})^{s-1 }xe^{x}\,dx
	<\frac{sp^2}{2}e^{-sp}e^p.
\end{split}
\]
Evaluating  \(e^p<1/(1-p )\) in the last estimate immediately leads to the upper bound of the second inequality.
Note that 
\[
\frac{d}{dx}\left(e^{2x}\frac{1-x}{1+x}\right)=-\frac{2 e^{2 x} x^2}{(x+1)^2}< 0,
\]
when \(x\not=0\) from which is the lower bound of the third inequality follows. For the upper bound
\[
\begin{split}
e^{-2sp}-\frac{(1-p)^s}{(1+p)^s}&=e^{-2sp}s\int_0^p\left(e^{2x}\frac{1-x}{1+x}\right)^{s-1}\frac{2 e^{2 x} x^2}{(x+1)^2}\,dx
\\
&<e^{-2sp}s\int_0^p\frac{2 e^{2 x} x^2}{(x+1)^2}\,dx
\\
&<
2e^{-2sp}s\int_0^p\frac{ x^2}{(1-x^2)^2}\,dx
\\
&<
\frac{2}{3}\frac{sp^3}{(1-p^2)^2}e^{-2sp},
\end{split}
\]
here we used the inequality \(e^{2x}<1/(1-x)^2\) to evaluate \(e^{2x}\) under the integration sign.
\end{proof}

\begin{lem} 
\label{lem_general_h}The function \(h_{\{1,2\}}^0(s,p)\)  satisfies the following inequalities:
\begin{equation}
\label{inequality_for_h_J}
	g(sp)\leqslant h^0_{\{1,2\}}(s,p)\leqslant g(sp)+\left(1+\frac{1}{\sqrt{e}}\right)\frac{p}{1-p}
\end{equation}
and 
\begin{equation}
\label{inequality_for_h_J_sym}
	\begin{split}
h^0_{J_{sym}}(s,p)\leqslant\frac{1}{2}\left( g(sp)+e^{-2sp}g(-sp)\right)	+\frac{1}{2}\left(1+\frac{2}{e^{3/4}}\right)\frac{p}{1-p}+\frac{4 e^{3/2}-1}{6 e^3}\frac{p^2}{(1-p^2)^2}
\end{split}
\end{equation}
for \(0<p<1\) and \(s\in \mathbb{Z}^+\),
where 
\[
\begin{split}
g(x)
&=1-2e^{-x}+\frac{1-e^{-x}}{x}.
\end{split}
\]

\end{lem}

\begin{proof} The expression (\ref{h_general_formula}) for \(h^0_{\{1,2\}}(s,p)\) leads to the identity
\[
	\begin{split}
			h^0_{\{1,2\}}(s,p)-g(sp)
			&= \frac{p}{1-p}+\bigl(e^{-sp}-(1-p)^s\bigr)\left(2+\frac{1}{sp}\right).
	\end{split}
	\]
Since the right-hand side of this identity is always positive by inequality (\ref{ineq_second}) of Lemma \ref{inequality_elementary}, the lower bound of inequality (\ref{inequality_for_h_J}) follows immediately. For the upper bound,  application of the inequality (\ref{ineq_second}) of  Lemma \ref{inequality_elementary} leads to the estimate
\begin{equation}
\label{ineq_J}
\begin{split}
			h^0_{\{1,2\}}(s,p)-g(sp)
			&\leqslant \frac{p}{1-p}\left(1+e^{-sp}\left(sp+\frac{1}{2}\right)\right).
	\end{split}	
\end{equation}
Since the function \(1+e^{-x}\left(x+\frac{1}{2}\right)\) has a unique maximum for \(x>0\) equal to \(1+\frac{1}{\sqrt{e}}\), the upper bound of the theorem follows.

Formula (\ref{sym_h_formula}) allows us to express \( h^0_{J_{sym}}(s,p) \) in terms of the functions \( h^0_{\{1,2\}}(s,p) \) and \( h^0_{\{1,2\}}(s,-p) \), while inequality (\ref{inequality_for_h_J}) suggests that \( h^0_{\{1,2\}}(s,p) \) and \( h^0_{\{1,2\}}(s,-p) \) can be approximated by \( g(sp) \) and \( g(-sp) \), respectively. With this in mind, we can write:
\[
\begin{split}
h^0_{J_{sym}}(s,p)
&=\frac{1}{2}\left( g(sp)+e^{-2sp}g(-sp)\right)+R_1+R_2+R_3
\end{split}
\]
where
\[
\begin{split}
	R_1&=\frac{1}{2}\left( h^0_{\{1,2\}}(s,p)-g(sp)\right)
	\\
	R_2&=\frac{1}{2}\left( \frac{(1-p)^s}{(1+p)^s}\bigl(h^0_{\{1,2\}}(s,-p)-g(-sp)\bigr)\right)
	\\
	R_3&=\frac{1}{2}\left( \frac{(1-p)^s}{(1+p)^s}-e^{-2sp}\right)g(-sp)
\end{split}
\]
Note that the quantity \( R_2 \) can be expressed by substituting the expressions for \( g(-sp) \) and \( h^0_{\{1,2\}}(s,-p) \) into it as
\[
R_2 = \frac{1}{2}\dfrac{\left( 1-p\right) ^{s}}{\left( 1+p\right) ^{s}}
\left( \dfrac{-p}{1+p}+\left( e^{sp}-\left( 1+p\right) ^{s}\right) \left( 2-\dfrac{1}{sp}\right) \right). 
\]
Dropping here the negative terms \(\dfrac{-p}{1+p}\) and \(\left( e^{sp}-\left( 1+p\right) ^{s}\right) \left(-\dfrac{1}{sp}\right)\) and making use of  inequalities (\ref{ineq_first}) and (\ref{ineq_third}) of the Lemma \ref{inequality_elementary} we get
\[
\begin{split}
R_2 &\leqslant \dfrac{\left( 1-p\right) ^{s}}{\left( 1+p\right) ^{s}}
 \left( e^{sp}-\left( 1+p\right) ^{s}\right)  
\leqslant e^{-2sp}
\frac{sp^2}{2}e^{sp} 
\leqslant \frac{sp^2}{2}e^{-sp}<\frac{sp^2e^{-sp}}{2(1-p)}.
\end{split}
\]
Note that  \(R_1\) can be readily evaluated using the inequality (\ref{ineq_J}), applying which together with the above estimate of \(R_2\)  we obtain the upper bound for the sum of these two quantities
\[
R_1+R_2\leqslant\frac{1}{2}\frac{p}{1-p}\left(1+e^{-sp}\left(2sp+\frac{1}{2}\right)\right)\leqslant \frac{1}{2}\left(1+\frac{2}{e^{3/4}}\right)\frac{p}{1-p},
\]
since \(\max_{x>0}\left(1+e^{-x}\left(2x+\frac{1}{2}\right)\right)=1+\frac{2}{e^{3/4}}.
\)

Since both the quantity \( g(-x) \) and its multiplier in the expression for \( R_3 \) are negative when \( x > 0 \), it follows that \( R_3 \) is positive. Applying inequality (\ref{ineq_third}) from Lemma (\ref{inequality_elementary}), we obtain:
\[
R_3<-\frac12e^{-2sp}g(-sp)\frac{2}{3}\frac{sp^3}{(1-p^2)^2}e^{-2sp}\leqslant\frac{4 e^{3/2}-1}{6 e^3}\frac{p^2}{(1-p^2)^2}
\]
since
\[
\max_{x>0}e^{-2x}(-g(x))x=\frac{4 e^{3/2}-1}{2 e^3}.
\]
By combining the obtained estimates for \( R_1, R_2, \) and \( R_3 \), we derive inequality (\ref{inequality_for_h_J_sym}) of the lemma.
\end{proof}

	\begin{figure}[ht]
    \centering
    \begin{tikzpicture}
    \begin{axis}[
        width=12cm, height=8cm,
        domain=0:20,
        samples=200,
        xmin=0, xmax=20,
        ymin=0, ymax=1.3,
        axis x line=bottom,
        axis y line=left,
        xlabel={$x$},
        ylabel={},
        legend pos=north east,
        legend style={draw=none},
        every axis plot/.append style={thick}
    ]
        \addplot [blue] {1 - exp(-x) + (1 - exp(-x)*(1 + x))/x} 
            node [pos=0.73, anchor=north west, blue] {$g(x)$};
        \addplot [red] {(1 - exp(-x) + (1 - exp(-x)*(1 + x))/x + exp(-2*x)*(1 - exp(x) + (1 - exp(x)*(1 - x))/(-x)))/2}
            node [pos=0.6, anchor=north west, red] {$\dfrac{g(x) + e^{-2x} g(-x)}{2}$};
    \end{axis}
\end{tikzpicture}
    \caption{Graph of $g(x)$ and $\bigr(g(x)+e^{-2x}g(-x)\bigr)/2$}
    \label{fig:example}
\end{figure}
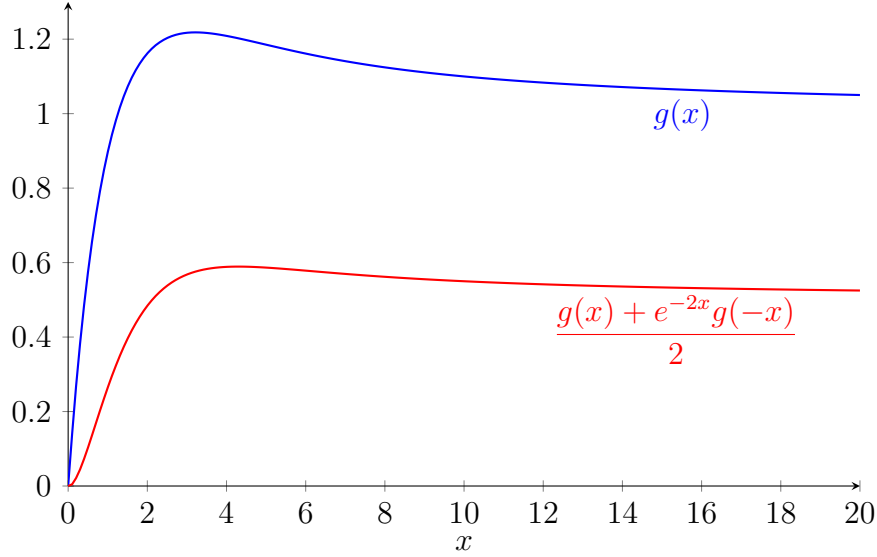
\begin{proof}[Proof of Lemma \ref{lem_C_estimates}] The proof follows immediately from the fact that \( h_J(s,p) \leqslant h^0_J(s,p) \), as established in Lemma \ref{lem_h0_vs_h}. 
This implies that \( C_J(p) \), defined as the maximum of \( h_J(s,p) \) over \( s \in [1,+\infty) \), cannot exceed the corresponding 
maximum of \( h^0_J(s,p) \). The inequalities in the lemma then follow directly from the upper bounds for \( h^0_J(s, p) \) 
provided by Lemma \ref{lem_general_h}, while also taking into account that \( g(x) \) and 
\( \bigl(g(x) + e^{-2x}g(-x)\bigr)/2 \) attain their unique maximum values of \( 1.21824\ldots \) and \( 0.58921\ldots \), 
respectively, on the interval \( (0, +\infty) \) (see Figure \ref{fig:example}). Furthermore, replacing the irrational constants involved 
 with their numerical upper bound estimates we complete the proof of both inequalities of the lemma.

For the computation of the items in the Table \ref{tab:rounded_estimates}
we used the exact expression for \(h_J(s,p)\).  One can easily show that function \(h(s,p)\) can be expressed as: 
\[
h_J(s,p)=\frac{(1-p)^s}{p^2}\left(\frac{p}{(1-p)^{s+1}} -\dfrac{s}{p^s} \int ^{p}_{0} \frac{t^s}{(1-t)^{s+1}}-\sum_{\substack{k\in J}}\frac{k^2}{(k+s)^2}\binom{k+s}{k}p^k\right).
\]
The integral term in this expression can be evaluated explicitly as:
\[
\int ^{p}_{0}\dfrac{t^{s}}{\left( 1-t\right) ^{s+1}}dt = \sum ^{s-1}_{j=0}\frac{\left( -1\right) ^{j+s+1}}{j+1}\left(\frac{p}{1-p}\right)^{j+1}+ (-1)^{s+1} \log(1-p)
\] 
when \(s\) is a positive integer.
This provides  an expression for \(h(s,p)\) in terms of finite sums and elementary functions.
The first row of the Table \ref{tab:rounded_estimates} was computed applying the resulting expression with \(J=\{1,2\}\). For the second row we expressed  \(h_{J_{sym}}(s,p)\) in terms of \(h_{J}(s,p)\) using the relation
\[
	h_{J_{sym}}(s,p)=\frac{1}{2}\left( h_{\{1,2\}}(s,p)+\frac{(1-p)^s}{(1+p)^s}h_{\{1,2\}}(s,-p)\right).
\]
\end{proof}

 \section*{Acknowledgments}
The author wrote this paper while serving as a visiting associate professor at the Institute of Statistical Sciences, Academia Sinica (Taiwan). He sincerely thanks Prof. Hsien-Kuei Hwang for his generous hospitality during the visit.

\end{document}